\documentclass[12pt]{amsart}
\usepackage{amsmath}
\usepackage{amsthm}
\usepackage{amssymb}
\usepackage{amsfonts,mathrsfs}
\usepackage{stmaryrd}
\usepackage{amsxtra}  
\usepackage{epsfig}
\usepackage{verbatim}

\theoremstyle{plain}
\newtheorem{theorem}[equation]{Theorem}
\newtheorem{proposition}[equation]{Proposition}
\newtheorem{lemma}[equation]{Lemma}
\newtheorem{corollary}[equation]{Corollary}
\newtheorem{definition}[equation]{Definition}
\theoremstyle{remark}
\newtheorem{remark}[equation]{Remark}
\newtheorem*{example}{Example}
\numberwithin{equation}{section}

\newcommand{\sjump}{\hskip .2 cm}

\begin{document}

\title[Convex defining functions]{Convex defining functions for convex domains}
\author{A.-K. Herbig \& J. D. McNeal}
\thanks{Research of the first author was supported by FWF grant P19147}
\thanks{Research  of the second author was partially supported by an NSF grant}
\address{Department of Mathematics, \newline University of Vienna, Vienna, 1090, Austria}
\email{anne-katrin.herbig@univie.ac.at}
\address{Department of Mathematics, \newline Ohio State University, Columbus, Ohio, USA}
\email{mcneal@math.ohio-state.edu}
\date{}
\begin{abstract} We give three proofs of the fact that a smoothly bounded, convex domain in $\mathbb{R}^n$ has defining functions whose Hessians are non-negative definite  in a neighborhood of the boundary of the domain.
 
\end{abstract}
\maketitle
\section{Introduction} 
Let $\Omega\subset\mathbb{R}^n$ be an open set with smooth boundary, $b\Omega$.
If $r$ is a $C^2$ defining function for $\Omega$ (see Definition \ref{D:definingfunc}), then $\Omega$ is convex if and only if

\begin{equation}\label{convexdomain_def}
\sum_{j,k=1}^n\frac{\partial^2 r}{\partial x_j\partial x_k}(p)\tau_j\tau_k\geq 0\quad\sjump\forall\sjump p\in b\Omega,\quad\forall\sjump\tau\in T_{p}\left(b\Omega\right).
\end{equation}
Note that convexity of $\Omega$ only forces (and requires) inequality \eqref{convexdomain_def} to hold for $p\in b\Omega$ and for directions $\tau$ tangent to $b\Omega$ at $p$.

Convex functions on open subsets of $\mathbb{R}^n$ that are sufficiently smooth are also characterized by non-negativity of their Hessians: if 
$f:U\to\mathbb{R}$ is of class $C^2$ on an open subset $U$ of $\mathbb{R}^{n}$, then $f$ is convex if and only if

\begin{equation}\label{convexfunc_def}
\sum_{j,k=1}^n\frac{\partial^2 f}{\partial x_j\partial x_k}(a)\xi_j\xi_k\geq 0\quad\sjump\forall\sjump a\in U,\quad\forall\sjump\xi\in \mathbb{R}^n.
\end{equation}
To emphasize the key point, note that inequality \eqref{convexfunc_def} is stipulated to hold for all points $a$ and in all directions $\xi\in \mathbb{R}^n$.

In this paper we show that \eqref{convexdomain_def} implies the existence of another defining function $\tilde r$ for $\Omega$ such that \eqref{convexfunc_def} holds, with $\tilde r$ in place of $f$. That is, smoothly bounded, convex domains in $\mathbb{R}^n$ possess defining functions that are actually convex in all directions, in a full neighborhood of $\overline\Omega$. 

The {\it existence} of a fully convex defining function for a smoothly bounded convex domain in $\mathbb{R}^n$ is not new. However previous proofs of this existence used geometric facts about convexity that are not true in intermediate, convexity-like situations (e.g., pseudoconvexity, weak linear convexity, $\mathbb{C}$-convexity, etc.). These proofs also gave little quantitative information about the new defining function.

The aim of this paper is to directly address how the positivity in condition \eqref{convexdomain_def} gives rise to defining functions satisfying \eqref{convexfunc_def}, avoiding use of additional facts about convexity. The fact that \eqref{convexdomain_def}  forces other defining functions for $\Omega$ to acquire the extra positivity expressed in  \eqref{convexfunc_def} is interesting from a purely analytic viewpoint. This gain in positivity is neither an obvious consequence of \eqref{convexdomain_def} nor does it follow from algebraic manipulations of this condition. Moreover, it does not hold for some other, natural non-negativity hypotheses that are similar to \eqref{convexdomain_def}. Domains of holomorphy in $\mathbb{C}^n$ with smooth boundary, for example, are characterized by a condition, {\it pseudoconvexity}, that is similar to  \eqref{convexdomain_def}: a domain $D\subset\mathbb{C}^n$ is pseudoconvex if
$$\sum_{j,k=1}^n\frac{\partial^2 \rho}{\partial z_j\partial\bar z_k}(p) w_j w_k\geq 0\quad\sjump\forall\sjump p\in bD,\quad\forall\sjump w\in \mathbb{C}T_{p}\left(bD\right),$$ where $\rho$ is a defining function for $D$. Note that pseudoconvexity, like \eqref{convexdomain_def},  is a tangential, non-negativity condition on the boundary involving second derivatives of a defining function.  However, there are smoothly bounded, pseudoconvex domains such that no defining function satisfies the analog of \eqref{convexfunc_def}, \cite{diederichfornaess77,fornaess79,behrens84}. 

We give three different arguments showing how \eqref{convexdomain_def} gives rise to a defining function satisfying \eqref{convexfunc_def}, resulting in three different, distinguished defining functions --- a local one obtained from solving a non-linear equation (Section \ref{S:IFT}), the signed distance-to-the-boundary function\footnote{This result is also proved by Gilbarg-Trudinger \cite{GilbargTrudinger}, pgs. 354--357, and by H\" ormander \cite{hormander_convex_book}, pgs. 57--60. We give another proof of the convexity of $\delta_{b\Omega}$ in Section \ref{S:delta}, see Corollaries \ref{C:dconvexonboundary} and 
\ref{C:dconvex}, from a somewhat different perspective than \cite{GilbargTrudinger} or \cite{hormander_convex_book}} (Section \ref{S:delta}), and a modification of an arbitrary defining function (Section \ref{S:other}). There are two different aspects to extracting the ``free'' positivity contained in \eqref{convexdomain_def}: (i) getting non-negativity for directions $\xi\notin T_p\left(b\Omega\right)$, and (ii) getting non-negativity for points $p\notin b\Omega$. In all three arguments, transforming the original defining function $r$ to a function $\tilde r$, which has (some of its) derivatives constant on $b\Omega$ (or on larger sets), is at the heart of the proof. This transformation leads to a certain control of the mixed terms in the Hessian of $\tilde r$, which is crucial to obtaining (i).

The primary virtue of all three proofs is that they yield estimates on the convex defining functions obtained. In applications, convex domains arise with an attendant defining function satisfying additional side conditions connected to the problem at hand. For example, the defining function may satisfy an auxiliary differential equation or be in a particular ``normal'' form. In order to combine the extra positivity of \eqref{convexfunc_def} with these side conditions, the estimates given by our proofs can be used to verify that the side conditions are inherited by the new defining function. Three arguments are given because they yield somewhat different estimates, and so are suitable for different applications of this type.

\medskip

\textbf{Acknowledgments.} We would like to thank H.-C. Herbig for a helpful remark on the proof of Theorem \ref{T:geomseries}. We are grateful for the stimulating atmosphere at the Erwin-Schr\"odinger International Institute for Mathematical Physics in Fall 2009 where this article was completed.

\section{Definitions and notations}

Throughout the paper, infinite differentiability of all functions and boundaries of sets that arise will be assumed, in order to avoid counting derivatives. This over-prescription of differentiability can easily be adjusted by the reader.

Let $\Omega\subset\mathbb{R}^{n}$ be an open set. The set $\Omega$ is said to be smoothly bounded if its topological boundary, $b\Omega$, is a $C^\infty$ manifold of dimension $n-1$. There are several equivalent ways to describe such sets; for our purposes, the most convenient way involves the following definition.

\begin{definition}\label{D:definingfunc} Let $\Omega\subset\mathbb{R}^{n}$ be an open set. A defining function for $\Omega$ is a function $r: \mathbb{R}^{n}\longrightarrow\mathbb{R}$ such that
$$\Omega=\left\{x\in\mathbb{R}^n: r(x)<0\right\}.$$ 

If $U\subset\mathbb{R}^n$ is open, then a function  $r: U \longrightarrow\mathbb{R}$ is called a local defining function for $\Omega$ on $U$ if $\Omega\cap U=\left\{x\in U: r(x) <0\right\}$ holds. 

$\Omega$ is smoothly bounded if it admits a defining function $r$ of class $C^\infty$, with $\nabla r\neq 0$ in a neighborhood of 
$\left\{ x\in\mathbb{R}^{n}:  r(x)=0\right\}= b\Omega$.
\end{definition}

There are many defining functions for a given smoothly bounded domain. Indeed, Definition \ref{D:definingfunc} shows that multiplying any defining function for a domain by a smooth, positive function in a neighborhood of $b\Omega$ gives another defining function for the domain. However this is the only flexibility present: if $r_1$ and $r_2$ locally define a common piece of a smoothly bounded open set $\Omega$, say on $U$, then there exists a $C^\infty$ multiplier $h: U\to\mathbb{R}^+$ such that
\begin{equation}\label{diff_def_func}
r_1(x)=h(x)\cdot r_2(x)\qquad\text{for } x\in U.
\end{equation}
The relationship \eqref{diff_def_func} requires a small argument since, after setting $h=r_1/r_2$ on $U\setminus b\Omega$, the function $r_2$ vanishes on $b\Omega$; see \cite{krantz_scv_book}, pgs. 114--115, for details.

Defining functions allow the tangent space to $b\Omega$ to be described by a single equation. If $\Omega$ is smoothly bounded, $x\in b\Omega$, and $T_{x}(b\Omega)$ denotes the tangent space to $b\Omega$ at $x$, then

\begin{equation}\label{tan_space}
T_{x}(b\Omega)=\left\{ \xi\in\mathbb{R}^n: \sum_{j=1}^n\frac{\partial r}{\partial x_j}(x)\xi_j=0\right\}.
\end{equation}
Note that \eqref{diff_def_func} implies that $\sum_{j=1}^n\frac{\partial r_1}{\partial x_j}(x)\xi_j=h(x)\cdot\sum_{j=1}^n\frac{\partial r_2}{\partial x_j}(x)\xi_j$ for any two local defining functions for $\Omega$ on some common open set containing $x\in b\Omega$, so the set on the right hand side of \eqref{tan_space} is independent of defining function.

\begin{definition}\label{Hess_f} 
Let $U\subset\mathbb{R}^{n}$ be an open set and let $f:U\longrightarrow\mathbb{R}$ be of class $C^\infty$.

(a)  The (real) Hessian of $f$ at $x\in U$ is the bilinear assignment
$$\left(\xi,\zeta\right)\longrightarrow\sum_{j,k=1}^{n}\frac{\partial^{2} f}{\partial x_{j}\partial x_{k}}(x)\xi_{j}\zeta_{k},$$
for $\xi,\zeta\in\mathbb{R}^n$. The expression
$$ H_{f}(\xi,\zeta)(x)=\sum_{j,k=1}^{n}\frac{\partial^{2} f}{\partial x_{j}\partial x_{k}}(x)\xi_{j}\zeta_{k}$$
will denote the action of the Hessian of $f$ at $x$ on the vectors $\xi,\zeta\in\mathbb{R}^n$. 

(b) The (real) Hessian matrix of  $f$ at $x\in U$ is the $n\times n$-matrix
\begin{align*}
  \mathcal{H}_{x}^{f}:=\left(\frac{\partial^{2}f}{\partial x_{j}\partial x_{k}}(x)\right)_{1\leq j,k \leq n}.
\end{align*}
\end{definition}
Two semi-definiteness conditions on $H_f$ restricted to the diagonal play a basic role in all that follows.

\begin{definition}\label{D:convex}

(a) If $\Omega\subset\mathbb{R}^n$ is a smoothly bounded open set, $U \subset\mathbb{R}^n$ is open with $p_0\in U\cap b\Omega$,
and $r$ is a local defining function for $\Omega$ on $U$, then $\Omega$ is convex near $p_0\in b\Omega$ if

\begin{equation}\label{convexdomain_Def}
H_r\left(\xi,\xi\right)(p)\geq 0\qquad \forall\sjump p\in U\cap b\Omega, \sjump\xi\in T_p(b\Omega).
\end{equation}

(b) If $U\subset\mathbb{R}^n$ is open and $f:U\longrightarrow\mathbb{R}$ is of class $C^\infty$, then $f$ is convex on  $U$ if

\begin{equation}\label{convexfunc_Def}
H_f\left(\xi,\xi\right)(x)\geq 0\qquad \forall\sjump x\in U, \sjump\xi\in\mathbb{R}^n.
\end{equation}
\end{definition}
We will use slightly non-standard terminology and say that a smooth defining $r$ for a smoothly bounded  domain $\Omega\subset\mathbb{R}^{n}$ is 
convex on  $b\Omega$ to mean
\begin{align*}
  H_{r}(\xi,\xi)(x)\geq 0\qquad\forall\sjump x\in b\Omega,\sjump\xi\in\mathbb{R}^{n}.
\end{align*}

\medskip

\begin{remark}\label{R:convexity}
(i) Under the smoothness hypothesis given in (a) of Definition \ref{D:convex}, the equivalence of \eqref{convexdomain_Def} with the classical, pointwise definition is shown in, e.g., Propositions 3.1.6 and 3.1.7, pg. 119, in \cite{krantz_scv_book}. That \eqref{convexfunc_Def} is equivalent to the classical definition for 
$f:U\subset\mathbb{R}^{n}\longrightarrow\mathbb{R}$ smooth follows from Corollary 1.1.10, pg. 6, in 
\cite{hormander_convex_book} after restricting $f$ to line segments contained in $U$.

\medskip

\noindent (ii) The invariance of \eqref{convexdomain_Def} under change of defining function follows from
$$H_{r_1}\left(\xi,\xi\right)(x)= h(x)\cdot H_{r_2}\left(\xi,\xi\right)(x),\qquad x\in b\Omega, \sjump\xi\in T_x\left(b\Omega\right),$$
which itself follows directly from \eqref{diff_def_func} by differentiation.

\medskip

\noindent (iii) Both \eqref{convexdomain_Def} and \eqref{convexfunc_Def} are preserved under affine changes of coordinates. This invariance will be used several times in the next sections. Indeed, if $A:\mathbb{R}^n\longrightarrow\mathbb{R}^{n}$ is an affine map, then
$$H_{f\circ A}\left(\xi,\xi\right)(x)= H_f\left(A^*\xi, A^*\xi\right)(Ax),$$
by straightforward computation. Here $A^*= D(A)$ is the derivative map associated to $A$. For \eqref{convexdomain_Def}, we also need
that $\xi\in T_x(b\Omega)$ implies that $A^*\xi\in T_{Ax}(b\Omega)$, which follows directly from \eqref{tan_space}.
\end{remark}

\bigskip

The following notation will simplify writing differential expressions in the next sections. If $f:U\subset\mathbb{R}^n\longrightarrow\mathbb{R}$ is smooth, the shorthand
$f_{x_j}=\frac{\partial f}{\partial x_j}$
will denote derivatives with respect to some given coordinates $\left(x_1,\dots, x_n\right)$ on $U$. If $v, w\in\mathbb{R}^n$, define the pairing 
$\left\langle v,w\right\rangle$ by
$\left\langle v,w\right\rangle =\sum_{j=1}^n v_j\, w_j,$  where the components of the vectors $v=(v_1,\dots, v_n)$ and $w=(w_1,\dots, w_n)$ are assumed to be with respect to the same basis for $\mathbb{R}^n$. In particular, if  $\xi\in\mathbb{R}^n$ and $f$ is as before, we write
$$ \left\langle\nabla f(x),\xi\right\rangle=\sum_{j=1}^n  f_{x_j}(x)\xi_{j},$$
with the understanding that the components of $\xi=(\xi_1,\dots, \xi_n)$ are those for the basis determined by the coordinates $\left(x_1,\dots, x_n\right)$ on $U$.

Finally, we will use the inequality $|ab|\leq \epsilon a^{2}+\frac{1}{4\epsilon}b^{2}$, $a,b\in\mathbb{R}$, $\epsilon>0$, and refer to it as the (sc)-(lc) inequality.

\section{An example}\label{S:Anexample}

The following example shows that, in general, obtaining the extra positivity mentioned in the introduction requires changing the defining function.

\begin{example} Consider the function $s:\mathbb{R}^2\longrightarrow\mathbb{R}$ given by
$$s(x,y)= y-y^2+ x^2.$$
Let $D$ be a domain locally defined by $s$ on a small neighborhood of $(0,0)$. Note that $(a,b)\in bD$ with $b <1$ implies that
\begin{equation}\label{ex_1}
b \leq0\qquad\text{ and }\qquad |a|=\sqrt{-b+b^2}. 
\end{equation}
A short computation yields
\begin{equation}\label{ex_2}
T_{(a,b)}\left(bD\right)=\left\{\xi=(\xi_1,\xi_2): 2a\xi_1+ (1-2b)\xi_2 =0\right\}
\end{equation}
and
\begin{equation}\label{ex_3}
H_s\left(\xi,\xi\right) = 2\left(\xi_1^2-\xi_2^2\right).
\end{equation}

If $(a,b)\in bD$ and $a\neq 0$, then \eqref{ex_2} shows that $\xi\in T_{(a,b)}\left(bD\right)$ implies that $\xi_1=\frac{2b-1}{2a}\xi_2$. It follows from \eqref{ex_1} that           $\left|\xi_1\right| > \left|\xi_2\right|$ in this case. If $a=0$, then \eqref{ex_2} forces $\xi_2=0$ for any $\xi\in T_{(a,b)}\left(bD\right)$, so $\left|\xi_1\right| \geq \left|\xi_2\right|$ in this case. It then follows from \eqref{ex_3} that $D$ is convex in a fixed neighborhood of $(0,0)$. However, \eqref{ex_3} also shows that $s$ is not convex in any neighborhood of the origin, since $H_s\left(\xi,\xi\right)  < 0$ for any direction with $\left|\xi_1\right|  < \left|\xi_2\right|$.\
\end{example}

Thus, (locally) convex open sets can have defining functions which are not (locally) convex. This obviously holds for globally convex sets as well.

\section{Via the Implicit Function Theorem}\label{S:IFT}

In this section, convex local defining functions for convex domains are constructed by employing the Implicit Function Theorem.

\begin{proposition}\label{P:locallyconvex}
 Let $\Omega\subset\mathbb{R}^{n}$ be a smoothly bounded, convex domain and $p\in b\Omega$. Then there exist a neighborhood $V\subset\mathbb{R}^{n}$ of $p$ and a smooth defining function $\rho$ for $\Omega$ on $V$ such that 
 $\rho$ is convex on $V$. 
\end{proposition}
\begin{proof}
Let $r$ be any smooth defining function for $\Omega$ near $p$. For $x\in\mathbb{R}^{n}$, write $x=(x',x_n)$ where $x':=(x_{1},\dots,x_{n-1})\in\mathbb{R}^{n-1}$ denotes the first $n-1$ components of $x$. Rotate the standard coordinates on $\mathbb{R}^n$, if necessary, to achieve  $r_{x_{n}}(p)>0$. This coordinate change preserves the convexity of $\Omega$, cf. Remark \ref{R:convexity} (iii).

The Implicit Function Theorem guarantees existence of a neighborhood $W\subset\mathbb{R}^{n-1}$ of $p'$ and a function $f\in C^{\infty}(W,\mathbb{R})$ such that $$r\left(x',f(x')\right)=0\qquad\forall\sjump x'\in W.$$ Moreover, the local graph of $f$ is precisely a piece of the hypersurface $b\Omega$: if $G=\left\{\left(x',f(x')\right): x'\in W\right\}$, then $G= b\Omega\cap V$ for some open set $V\subset\mathbb{R}^{n}$ containing $p$. Since $r_{x_{n}}(p)>0$, it follows that
\begin{align}\label{E:definitionrho}
  \rho(x):=x_{n}-f(x'),\qquad x\in V,
\end{align}  
is a smooth defining function for $\Omega$ on $V$. 

 The linearity of $\rho$ in the $x_{n}$-variable implies that $\rho_{x_{n}x_{\ell}}(x)=0$ for all $x\in V$ and $\ell\in\{1,\dots,n\}$. Therefore the Hessian of $\rho$ in a direction $\xi$ depends only on the vector's first $n-1$ components: if $\xi=(\xi',\xi_n)\in\mathbb{R}^{n}$, then 
  \begin{align}\label{E:noxin}
    H_{\rho}(\xi,\xi)(x)=H_{\rho}\bigl((\xi',A),(\xi',B)\bigr)(x)\qquad\sjump\forall\sjump x\in V,
  \end{align}
  for arbitrary $A,B\in\mathbb{R}$.
  
  Now fix $x\in b\Omega\cap V$ and consider an arbitrary direction $\xi\in\mathbb{R}^n$. From \eqref{tan_space} and the fact that $\rho_{x_{n}}(x)=1$, it follows that
  $$\left(\xi', A\right)\in T_{x}(b\Omega)\quad\text{ if }\quad A= -\sum_{j=1}^{n-1}\rho_{x_{j}}(x)\xi_{j}.$$
  Define $\tau=(\xi', A)$ for this value of $A$. Since $\Omega$ is convex, $H_{\rho}(\tau,\tau)(x)\geq 0$, from which it follows by \eqref{E:noxin} that  $H_{\rho}(\xi,\xi)(x)\geq 0$. Because
  $x\in b\Omega\cap V$ and $\xi\in\mathbb{R}^{n}$ were arbitrary, we conclude that $\rho$ is convex on $b\Omega\cap V$.
  
  To see that $\rho$ is also convex off $b\Omega\cap V$,  observe that the entries of the Hessian matrix of $\rho$ ---
  $\rho_{x_{i}x_{j}}(x)$, $1\leq i,j\leq n$ --- are independent of $x_{n}$. In particular,  
  \begin{align}\label{E:noxin2}
   \mathcal{H}_{(x',x_n)}^{\rho}= \mathcal{H}_{(x',f(x'))}^{\rho}\qquad\forall\sjump x\in V.
  \end{align}
 Since $(x',f(x'))\in b\Omega\cap V$, we have that the matrix $\mathcal{H}_{(x',f(x'))}^{\rho}$ is non-negative definite. Identity \eqref{E:noxin2} then shows that $\rho$ 
 is convex on $V$.
\end{proof}

\medskip

\begin{remark} (i) Derivatives of the defining function $\rho$ given by \eqref{E:definitionrho} can be expressed in terms of derivatives of the original defining function $r$. For example, differentiating 
the equation $r\left(x',f(x')\right)=0$ with respect to $x_j$ gives
\begin{align*}
  r_{x_{j}}\left(x',f(x')\right)+r_{x_{n}}\left(x',f(x')\right)\cdot f_{x_{j}}(x')=0\qquad\forall\sjump j\in\{1,\dots,n-1\}, 
\end{align*}
and from this it follows
\begin{align*}
  \rho_{x_{j}}(x)=-f_{x_{j}}(x')=\frac{r_{x_{j}}\left(x',f(x')\right)}{r_{x_{n}}\left(x',f(x')\right)}\qquad\forall
 \sjump x\in V.
\end{align*}
 The second derivatives are obtained by further computation:
 \begin{align*}
    \rho_{x_{j}x_{k}}(x)=\left(\frac{1}{r_{x_{n}}}\left[
    r_{x_{j}x_{k}}-r_{x_{j}x_{n}}\frac{r_{x_{k}}}{r_{x_{n}}}-r_{x_{k}x_{n}}\frac{r_{x_{j}}}{r_{x_{n}}}
    +r_{x_{n}x_{n}}\frac{r_{x_{j}}}{r_{x_{n}}}\frac{r_{x_{k}}}{r_{x_{n}}}
    \right]\right)(x',f(x'))
 \end{align*}   
 for $j,k\in\{1,\dots,n-1\}$ and $x\in V$. Higher derivatives are obtained similarly.
 \medskip
 
\noindent (ii)  The content of Proposition \ref{P:locallyconvex} can be succinctly expressed as follows: define the linear map
 $$T=T_{x}:\mathbb{R}^{n}\longrightarrow T_{x}(b\Omega)\;\; \text{ with }\;\;
 T(\xi)=\bigl(\xi_{1},\dots,\xi_{n-1},-\sum_{j=1}^{n-1}\rho_{x_{j}}(x)\xi_{j}\bigr).$$
 Then the Hessian of the defining function given by \eqref{E:definitionrho} satisfies
 \begin{align*}
   H_{\rho}(\xi,\xi)(x)=\frac{1}{r_{x_{n}}(x',f(x'))}\cdot H_{r}\bigl(T(\xi),T(\xi)\bigr) \left(x',f(x')\right)\qquad\forall\sjump x\in V.
 \end{align*}\
\end{remark} 
\medskip

\vskip1cm

\section{The signed distance-to-the-boundary function}\label{S:delta}

For $\Omega\subset\mathbb{R}^{n}$ a smoothly bounded domain, let  $d_{b\Omega}(x)=\inf\{\|x-z\|:z\in b\Omega\}$ be the Euclidean distance of $x$ to $b\Omega$. Define the signed distance-to-the-boundary function $\delta_{b\Omega}(x)$ by
\begin{equation*}
  \delta_{b\Omega}(x)=
  \begin{cases}
    -d_{b\Omega}(x), & x\in \Omega\\
    \;\;\;d_{b\Omega}(x), & x\in\mathbb{R}^{n}\setminus\Omega
  \end{cases}\;\;.
\end{equation*}

To conclude that $\delta_{b\Omega}$ is a defining function for $\Omega$ which is smooth in a neighborhood of $b\Omega$ we use the following facts.

\begin{lemma}\label{L:facts}
If $\Omega\subset\subset\mathbb{R}^{n}$ is a smoothly bounded domain, then there exists a neighborhood $U$ of $b\Omega$ such that:
\begin{itemize}
  \item[(a)] The map $b_{b\Omega}: U\longrightarrow b\Omega$ satisfying $\|b_{b\Omega}(x)-x\|=
  |\delta_{b\Omega}(x)|$  is well-defined.
  \item[(b)] Both $d_{b\Omega}(x)$ and $b_{b\Omega}(x)$ are smooth on $\overline{\Omega}\cap U$ and
  on $(\mathbb{R}^{n}\setminus\Omega)\cap U$.
  \item[(c)] For $p\in b\Omega$, let $\nu_{p}$ be the outward unit normal vector to $b\Omega$ at $p$. Then
  \begin{equation}
  \nabla d_{b\Omega}(x)=
  \begin{cases}
    -\nu_{b_{b\Omega}(x)}, & x\in \Omega\cap U\\
    \;\;\;\nu_{b_{b\Omega}(x)}, & x\in(\mathbb{R}^{n}\setminus\overline{\Omega})\cap U
  \end{cases}\;\;.
\end{equation}
\end{itemize}
\end{lemma}
For a proof of (a) see, e.g., Lemma 4.1.1., pgs. 444--445, in \cite{federer59}. Parts (b) and (c) follow from Lemma 1, pg. 382, in \cite{GilbargTrudinger}. When there is no reason for confusion we shall drop the subscripts of $d_{b\Omega}$, $\delta_{b\Omega}$ and $b_{b\Omega}$.

\begin{corollary}\label{C:deltaderivatives}
  Let $\Omega\subset\subset\mathbb{R}^{n}$ be a smoothly bounded domain. 
  Then there exists a neighborhood $U\subset\mathbb{R}^{n}$ of $b\Omega$ such that
  $\delta_{b\Omega}$ and $b_{b\Omega}$ are smooth on $U$ and
  \begin{align}\label{E:deltanu}
    \nabla\delta_{b\Omega}(x)=\nu_{b_{b\Omega}(x)}\qquad\sjump\forall\sjump x\in U.
  \end{align}
\end{corollary}
\begin{proof}
  Let $U\subset\mathbb{R}^{n}$ be a neighborhood of $b\Omega$ such that (a)-(c) of Lemma
  \ref{L:facts} hold. Then it follows from the definition of $\delta$ and (b) that $\delta$ is smooth on
  $\overline{\Omega}\cap U$ and $(\mathbb{R}^{n}\setminus\Omega)\cap U$. Moreover, (c) then implies 
  that
  \begin{align}\label{E:deltanugezeigt}
   \nabla\delta(x)=\nu_{b(x)}\qquad\sjump\forall\sjump x\in U.
  \end{align}
  Note that, if $r$ is a smooth defining function for $\Omega$, then
  $\nu_{x}=\frac{\nabla r(x)}{\|\nabla r(x)\|}$ for $x\in b\Omega$. Thus $\nu_{x}$ extends to a smooth function in a neighborhood of $b\Omega$,
   which implies that $\delta\in C^{1}(U)$.
  Next, note that for given $x\in U$ the function $\|y-x\|^{2}$ subject to the constraint $r(y)=0$ attains its 
  minimum value on $U$ at $b(x)$. This implies that the vector $b(x)-x$ is parallel to $\nu_{b(x)}$. Since
  $\|b(x)-x\|=|\delta(x)|$, it then follows that $b(x)-x=-\delta(x)\nu_{b(x)}$. Hence, by \eqref{E:deltanugezeigt},
  \begin{align}\label{E:bidentity}
    b(x)=x-\delta(x)\nabla\delta(x)\qquad\sjump\forall\sjump x\in U.
  \end{align}
  The fact that $\delta$ is in $C^{1}(U)$ and smooth on both $\overline{\Omega}\cap U$ and 
  $(\mathbb{R}^{n}\setminus\Omega)\cap U$ forces then $b$ to be of class $C^{1}$ on $U$. 
  However, it then follows from \eqref{E:deltanugezeigt} that $\delta\in C^{2}(U)$. Proceeding inductively 
  completes the proof.
 \end{proof}

\medskip

\begin{corollary}\label{C:dconvexonboundary}
  Let $\Omega\subset\mathbb{R}^{n}$ be a smoothly bounded domain and $p\in b\Omega$. Suppose 
  $\Omega$ is convex in a neighborhood $U\subset\mathbb{R}^{n}$ of $p$. Then $\delta_{b\Omega}$ is convex on $b\Omega\cap U$.
\end{corollary}
\begin{proof}
   Shrink $U$ so that $\delta$ is smooth on $U$. Then
   \eqref{E:deltanu} implies that $\|\nabla\delta(x)\|^{2}=1$ for all $x\in U$.
   Differentiating this equation
   in a direction $\xi\in\mathbb{R}^{n}$ yields
   \begin{align}\label{E:pairingnabladelta}
     0=\sum_{k=1}^{n}\xi_{k}\frac{\partial}{\partial x_{k}}\Bigl(\bigl\|\nabla\delta(x)\bigr\|^{2}\Bigr)
     &=2\sum_{j,k=1}^{n}\delta_{x_{j}x_{k}}(x)\delta_{x_{j}}(x)\xi_{k}\notag\\
     &=2H_{\delta}(\xi,\nabla\delta)(x)\qquad\qquad\sjump\forall\sjump x\in U. 
   \end{align}
  For given $x\in b\Omega$ and $\xi\in\mathbb{R}^{n}$, let 
   $\xi^{T}$ be the orthogonal projection of $\xi$ on $T_{x}(b\Omega)$ and set $\xi^{N}=\xi-\xi^{T}$. Then 
   it follows that $\xi^{N}$ is parallel to $\nabla\delta(x)$. Thus, \eqref{E:pairingnabladelta} implies that
   \begin{align*}
     H_{\delta}(\xi^{T},\xi^{N})(x)=0=H_{\delta}(\xi^{N},\xi^{N})(x). 
   \end{align*}
   Therefore, we obtain
   \begin{align*}
     H_{\delta}(\xi,\xi)(x)&=H_{\delta}\left(\xi^{T},\xi^{T}\right)(x)+2H_{\delta}\left(\xi^{T},\xi^{N}\right)(x)
     +H_{\delta}\left(\xi^{N},\xi^{N}\right)(x)\\
     &=H_{\delta}\left(\xi^{T},\xi^{T}\right)(x)\geq 0,
   \end{align*}
   since $\Omega$ is convex on $U$.      
  \end{proof}

  Corollary \ref{C:deltaderivatives}  also implies that the Hessian matrix $\mathcal{H}_{x}^{\delta}$ 
  is related to $\mathcal{H}_{b(x)}^{\delta}$ via a geometric series.
   
  \begin{theorem}\label{T:geomseries}
    Let $\Omega\subset\subset\mathbb{R}^{n}$ be a smoothly bounded domain. Then there exists a 
    neighborhood
    $U\subset\mathbb{R}^{n}$ of $b\Omega$ such that
    \begin{align*}
      \mathcal{H}_{x}^{\delta} =
      \mathcal{H}_{b(x)}^{\delta}\cdot\left(\sum_{m=0}^{\infty}\left(-\delta(x)\right)^{m}
      \left(\mathcal{H}_{b(x)}^{\delta}\right)^{m}\right)
      =\mathcal{H}^{\delta}_{b(x)}\cdot\left(\text{I}+\delta(x)\mathcal{H}_{b(x)}^{\delta}\right)^{-1}
    \end{align*}
    for all $x\in U$. Here, $I$ is the $n\times n$ identity matrix.
  \end{theorem}
  
  Theorem \ref{T:geomseries} was proved by Weinstock in \cite{Weinstock75}, pgs. 402--403, and by Gilbarg--Trudinger  in 
  \cite{GilbargTrudinger}, pgs. 354--357, though it was not stated in this form.
  
  \begin{proof}
    Let $U\subset\mathbb{R}^{n}$ be a neighborhood of $b\Omega$ such that Corollary
    \ref{C:deltaderivatives} holds. Identity \eqref{E:deltanu} implies that
     \begin{align}\label{E:deltaderivatives}
      \delta_{x_{j}}(x)=\delta_{x_{j}}\bigl(b(x)\bigr)\qquad\forall\sjump x\in U,\sjump j\in\{1,\dots,n\}.
    \end{align}
    Differentiating these equations with respect to $x_{k}$, $1\leq k\leq n$, yields
    \begin{align}\label{E:secondderivativesdelta}
      \delta_{x_{j}x_{k}}(x)=\sum_{\ell=1}^{n}\delta_{x_{j}x_{\ell}}\bigl(b(x)\bigr)
      \frac{\partial b_{\ell}}{\partial x_{k}}(x).
    \end{align}
    Identity \eqref{E:bidentity} gives
    $b_{\ell}(x)=x_{\ell}-\delta(x)\delta_{x_{\ell}}(x)$
    for all $\ell\in\{1,\dots,n\}$.
        Differentiating these equations with respect to $x_{k}$ for $k\in\{1,\dots,n\}$ results in
    \begin{align*}
        \frac{\partial b_{\ell}}{\partial x_{k}}(x)=\delta_{k}^{\ell}-\delta_{x_{k}}(x)\delta_{x_{\ell}}(x)-\delta(x)
        \delta_{x_{\ell}x_{k}}(x)\qquad\forall\sjump\ell,k\in\{1,\dots,n\},
     \end{align*}
     which, plugged  into \eqref{E:secondderivativesdelta}, yields
     \begin{align*}
        \delta_{x_{j}x_{k}}(x)= \delta_{x_{j}x_{k}}\bigl(b(x)\bigr)
        &-\sum_{\ell=1}^{n} \delta_{x_{j}x_{\ell}}\bigl(b(x)\bigr)\delta_{x_{\ell}}(x)\delta_{x_{k}}(x)\\
        &\hspace{2cm}-\delta(x)\sum_{\ell=1}^{n}\delta_{x_{j}x_{\ell}}\bigl(b(x)\bigr)\delta_{x_{\ell}x_{k}}(x)
     \end{align*}
     for all $j,k\in\{1,\dots,n\}$. 
     However, the second term on the right hand side vanishes: 
     \eqref{E:pairingnabladelta} and \eqref{E:deltaderivatives} imply
     \begin{align*}
     0=\sum_{\ell=1}^{n}\delta_{x_{j}x_{\ell}}\bigl(b(x)\bigr)\delta_{x_{\ell}}\bigl(b(x)\bigr)
     =\sum_{\ell=1}^{n}\delta_{x_{j}x_{\ell}}\bigl(b(x)\bigr)\delta_{x_{\ell}}(x)\qquad\forall\sjump j\in\{1,\dots,n\}.
     \end{align*}
     Thus, it follows that
     \begin{align*}
        \delta_{x_{j}x_{k}}(x)= \delta_{x_{j}x_{k}}\bigl(b(x)\bigr)
        -\delta(x)\sum_{\ell=1}^{n}\delta_{x_{j}x_{\ell}}\bigl(b(x)\bigr)\delta_{x_{\ell}x_{k}}(x)\qquad\forall\sjump j,k\in\{1,\dots,n\}.
     \end{align*}
     In matrix form, this says $\mathcal{H}_{x}^{\delta}
     =\mathcal{H}_{b(x)}^{\delta}-\delta(x)\mathcal{H}_{b(x)}^{\delta}
       \cdot\mathcal{H}_{x}^{\delta}$.
     Thus
     \begin{align*}
       \mathcal{H}_{b(x)}^{\delta}=\mathcal{H}_{x}^{\delta}+\delta(x)\mathcal{H}_{b(x)}^{\delta}
       \cdot\mathcal{H}_{x}^{\delta}=\left(
       I+\delta(x)\mathcal{H}_{b(x)}^{\delta}
       \right)\mathcal{H}_{x}^{\delta},
     \end{align*}
     which leads to the claimed identity.
\end{proof}

\begin{corollary}\label{C:dconvex}
  Let $\Omega\subset\mathbb{R}^{n}$ be a smoothly bounded domain and $p\in b\Omega$. Let 
  $U\subset\mathbb{R}^{n}$ be a neighborhood of $p$ such that Corollary \ref{C:deltaderivatives} holds.
  If $\Omega$ is convex on $U$, then $\delta_{b\Omega}$ is convex on $U$.
\end{corollary}
\begin{proof}
By Corollary \ref{C:dconvexonboundary} $\delta$ is convex on $b\Omega\cap U$.
So let us  consider  $x\in U\setminus b\Omega$ and $\xi\in\mathbb{R}^{n}$. Theorem \ref{T:geomseries} yields
\begin{align}\label{E:Hessianformgeometric}
 H_{\delta}(\xi,\xi)(x)=H_{\delta}(\xi,\xi)\bigl(b(x)\bigr)
 &+\bigl(-\delta(x)\bigr)\left\|\mathcal{H}_{b(x)}^{\delta}\xi\right\|^2\\
 &+\bigl(-\delta(x)\bigr)^2H_{\delta}\left(\mathcal{H}_{b(x)}^{\delta}\xi,\mathcal{H}_{b(x)}^{\delta}\xi\right)(x)+\dots .\notag
\end{align} 
Since $\delta$ is convex on the boundary, it follows that for all $x\in\Omega\cap U$ all terms on the right hand side of \eqref{E:Hessianformgeometric}
are non-negative. Thus $\delta$ is convex on $\Omega\cap U$.

Now consider $x\in(\mathbb{R}^{n}\setminus\overline{\Omega})\cap U$. Because of the convexity of $\delta$ on $b\Omega$, a Cauchy--Schwarz inequality holds for the Hessian of $\delta$ on $b\Omega$, i.e.,
\begin{align*}
  \left|H_{\delta}(\zeta,\eta)(p)\right|\leq\left(H_{\delta}(\zeta,\zeta)(p)\right)^{\frac{1}{2}}\cdot
  \left(H_{\delta}(\eta,\eta)(p)\right)^{\frac{1}{2}}\;\;\;\;\;\forall\;p\in b\Omega,\; \eta,\zeta\in\mathbb{R}^{n}.
\end{align*}
Since $0\leq\|\mathcal{H}_{b(x)}^{\delta}\xi\|^{2}=H_{\delta}\bigl(\xi,\mathcal{H}_{b(x)}^{\delta}\xi\bigr)
\bigl(b(x)\bigr)$, it follows that
\begin{align*}
  -\delta(x)\left\|\mathcal{H}_{b(x)}^{\delta}\xi\right\|^{2}&\geq -\delta(x)
  \Bigl[H_{\delta}(\xi,\xi)\bigl(b(x)\bigr)\Bigr]^{\frac{1}{2}}
  \cdot 
  \Bigl[H_{\delta}\bigl(\mathcal{H}_{b(x)}^{\delta}\xi,
  \mathcal{H}_{b(x)}^{\delta}\xi\bigr)\bigl(b(x)\bigr)\Bigr]^{\frac{1}{2}}\\
  &\geq
   -\frac{1}{2}H_{\delta}(\xi,\xi)\bigl(b(x)\bigr)
   -\frac{\delta^{2}(x)}{2}
   H_{\delta}\bigl(\mathcal{H}_{b(x)}^{\delta}\xi,\mathcal{H}_{b(x)}^{\delta}\xi\bigr)\bigl(b(x)\bigr),
\end{align*}
where the (sc)-(lc) inequality was used in the last step. Thus \eqref{E:Hessianformgeometric} becomes
\begin{align*}
   &H_{\delta}(\xi,\xi)(x)\geq \frac{1}{2}H_{\delta}(\xi,\xi)\bigl(b(x)\bigr)+
   \frac{\delta^{2}(x)}{2}H_{\delta}\left(\mathcal{H}_{b(x)}^{\delta}\xi,\mathcal{H}_{b(x)}^{\delta}\xi\right)
   \bigl(b(x)\bigr)\\
   &\hspace{1cm}+
   \sum_{m=2}^{\infty}\left[
   -\delta^{2m-1}(x)\left\|\left(\mathcal{H}_{b(x)}^{\delta}\right)^{m}\xi\right\|^{2}
   +\delta^{2m}(x)H_{\delta}\left(
   \left(\mathcal{H}_{b(x)}^{\delta}\right)^{m}\xi,\left(\mathcal{H}_{b(x)}^{\delta}\right)^{m}\xi
   \right)(b(x))
   \right]
\end{align*}
Performing the analogous arguments for the terms of the form 
$(-\delta(x))^{2m-1}\|(\mathcal{H}_{b(x)}^{\delta})^{m}\xi\|^{2}$ for integer $m\geq 2$, then leads to
\begin{align*}
  H_{\delta}(\xi,\xi)(x)\geq\frac{1}{2}H_{\delta}(\xi,\xi)\bigl(b(x)\bigr).
\end{align*}
Hence, $\delta$ is also convex on $(\mathbb{R}^{n}\setminus\overline{\Omega})\cap U$.
\end{proof}

\section{Other defining functions}\label{S:other}

In practice, convex domains usually arise with an attendant defining function satisfying some additional conditions.  In these situations it is undesirable to consider the defining function $\rho$ given in Section \ref{S:IFT} or the function $\delta$ in Section \ref{S:delta}, especially as these functions are difficult to explicitly write down. 
In this section we show how to transform an arbitrary defining function, of a given smoothly bounded, convex domain, into a convex defining function
by a sequence of direct modifications.
\vskip .5cm

Let $\Omega\subset\subset\mathbb{R}^n$ be a smoothly bounded, convex domain and $r$ a smooth defining function for $\Omega$. It is convenient to split the space of directions $\xi\in\mathbb{R}^n$ using the subspace $T_x\left(b\Omega\right)$. For $x\in b\Omega$ and $\xi\in\mathbb{R}^n$, define
$$\xi^N=\frac{\left\langle\nabla r(x),\xi\right\rangle}{\left\|\nabla r(x)\right\|^2}\: \nabla r(x),$$
and set $\xi^T=\xi -\xi^N$. Thus $\xi=\xi^{T}+\xi^{N}$, and it is easy to check that 
\begin{enumerate}
\item[(i)] $\xi^{T}\in T_x(b\Omega)$, and 
\item[(ii)] $\xi^{N}$ is orthogonal to $T_x (b\Omega)$.
\end{enumerate}
This decomposition depends on $x\in b\Omega$ (though not on the defining function $r$, because of \eqref{diff_def_func}). When it serves to clarify the arguments below, the dependence on $x$ will either be denoted --- $\xi^T_x, \xi^N_x$ --- or explicitly mentioned.

Bilinearity gives the following expansion of the Hessian:
\begin{align}\label{bilinear}
H_r\left(\xi,\xi\right)(x)&= H_r\left(\xi^T,\xi^T\right)(x)+2 H_r\left(\xi^T,\xi^N\right)(x)+H_r\left(\xi^N,\xi^N\right)(x) \\
&= H_r\left(\xi^T,\xi^T\right)(x)+2\alpha\: H_r\left(\xi^T,\nabla r\right)(x)+\alpha^2\: H_r\left(\nabla r,\nabla r\right)(x) \notag \\
&=\mathcal{T}+\mathcal{M}+\mathcal{N},\notag
\end{align}
where $\alpha=\frac{\left\langle\nabla r(x),\xi\right\rangle}{\left\|\nabla r\right\|^2}$.
If $x\in b\Omega$, and $\Omega$ is convex, then automatically $\mathcal{T}\geq 0$. On the other hand, the terms $\mathcal{M}$ and 
$\mathcal{N}$ may be negative, even when $x\in b\Omega$. Favorably estimating these terms, by adjusting $r$, is required
to conclude that $\mathcal{T}+\mathcal{M}+\mathcal{N}\geq 0$ even for $x\in b\Omega$. After that, further adjustments of $r$ will be needed to get
non-negativity of \eqref{bilinear} off $b\Omega$.

While seeking appropriate estimates on $\mathcal{M}$ and $\mathcal{N}$, \eqref{diff_def_func} circumscribes the allowed modifications of $r$: any defining function 
$\tilde r$ is of the form $\tilde r = h\cdot r$, for $h> 0$ and $C^\infty$ near $b\Omega$. All such modifications preserve the non-negativity of $\mathcal{T}$, cf.\eqref{tan_space}. However for $\xi\in\mathbb{R}^n$ arbitrary, note that

\begin{align}\label{hessian_hXr}
H_{\tilde r}\left(\xi,\xi\right)(x)=&\,\, h(x)\cdot H_r\left(\xi,\xi\right)(x) +2\left\langle \nabla h(x),\xi\right\rangle \left\langle \nabla r(x),\xi\right\rangle + r(x)\cdot H_h\left(\xi,\xi\right)(x) \notag \\
=&\,\, h(x)\cdot H_r\left(\xi,\xi\right)(x) +2\left\langle \nabla h(x),\xi\right\rangle \left\langle \nabla r(x),\xi\right\rangle,\quad\text{if } x\in b\Omega .
\end{align} 

The Hessians of functional combinations of $r$ will also occur below. If $\chi:\mathbb{R}\to\mathbb{R}$ is $C^2$ and $\hat r=\chi\circ r$, then
\begin{equation}\label{hessian_chiOr}
H_{\hat r}\left(\xi,\xi\right)(x)=\chi'(r)\cdot H_{r}\left(\xi,\xi\right)(x) +\chi''(r)\left|\left<\nabla r(x),\xi\right>\right|^2.
\end{equation}
\vskip .25cm

\subsection{Convexity on $b\Omega$}\label{SS:convexonboundary}

\subsubsection{Controlling $\mathcal M$}\label{SS:controllingM}

Consider the function $r_0=\frac r{\left\|\nabla r\right\|}$, i.e., $r_0=h^0\cdot r$ for $h^0=\frac 1{\left\|\nabla r\right\|}$. Then $\left\|\nabla r_{0}\right\|^{2}=1$ on $b\Omega$. Differentiating this equation in  tangential directions gives useful information. Let $X=\sum_{i=1}^{n}\tau_{k}\frac{\partial}{\partial x_{k}}$ be a vector field with $\tau\in T_{x}(b\Omega)$ for a given $x\in b\Omega$. Then
\begin{align*}
  0=X\left(\|\nabla r_{0}\|^{2}\right)(x)=
  2H_{r_0}\left(\tau,\nabla r_0\right)(x)=0.
\end{align*}
It then follows from \eqref{bilinear} that
\begin{align}\label{E:nomixedterms}
H_{r_0}\left(\xi,\xi\right)(x)= H_{r_0}\left(\xi^T,\xi^T\right)(x)+H_{r_0}\left(\xi^N,\xi^N\right)(x),
\end{align}
i.e., the mixed terms in the Hessian of $r_0$ vanish on the boundary.

We want to re-express this Hessian in terms of $r$. First notice that, for $h^0=\frac 1{\left\|\nabla r\right\|}$,
\begin{equation*}
\left|\left< \nabla h^0(x),\xi\right>\right| = -h^0(x)\cdot H_r\left(\frac{\nabla r}{\left\|\nabla r\right\|^2},\xi\right)(x).
\end{equation*}
It then follows from \eqref{hessian_hXr} that
\begin{equation}\label{hessian_r0}
  H_{r_0}(\xi,\xi)(x)=\frac{1}{\|\nabla r(x)\|}
    H_{r}\left(\xi^{T},\xi^{T}\right)(x)- \frac{1}{\|\nabla r(x)\|}H_{r}\left(\xi^{N},\xi^{N}\right)(x)
\end{equation}
holds, after noting that $\left<\nabla r(x),\xi^T\right>=0$. The passage from $r$ to $r_0$ therefore completely eliminates the term $\mathcal M$, while only changing the sign of the term $\mathcal N$, in \eqref{bilinear}.

\subsubsection{Controlling $\mathcal N$}\label{SS:controllingN}

This modification is standard. Consider the function $r_1=r_0+Kr_0^2$, for large $K>0$ to be chosen, i.e., $r_1=h^1\cdot r_0$ for $h^1=1+Kr_0$. It is convenient to compute $H_{r_1}$ additively, using \eqref{hessian_chiOr}, rather than by using \eqref{hessian_hXr}. First, note that if $\chi(x)=Kx^2$ and $\hat r=\chi\circ r_0$, then
\eqref{hessian_chiOr} implies
\begin{align*}
H_{\hat r}\left(\xi,\xi\right)(x)&= 2Kr_0(x)H_{r_0}\left(\xi,\xi\right)+2K\left|\left<\nabla r_0(x),\xi\right>\right|^2 \\
&= 2K\left|\left<\nabla r_0(x),\xi\right>\right|^2 ,\qquad\text{if }x\in b\Omega.
\end{align*}
Second, observe that
$$H_{r_0}\left(\xi^N,\xi^N\right)(x)=\mathcal{O}\left(\left|\left<\nabla r_0(x),\xi\right>\right|^2\right).$$ Therefore, choosing $K>\max\left\{0,-\|\nabla r(x)\|^{-1}H_r\left(\nabla r_{0},\nabla r_{0}\right)(x):x\in b\Omega\right\}$, it follows from \eqref{hessian_r0} that
\begin{align}\label{hessian_r1}
H_{r_1}\left(\xi,\xi\right)(x)= H_{r_0}\left(\xi,\xi\right)(x)+H_{\hat r}\left(\xi,\xi\right)(&x) \notag\\
= \|\nabla r(x)\|^{-1}
    H_{r}\left(\xi^{T},\xi^{T}\right)(x) &- \|\nabla r(x)\|^{-1}H_{r}\left(\xi^{N},\xi^{N}\right)(x)\notag
    \\ &+ 2K\left|\left<\nabla r_0(x),\xi\right>\right|^2\notag \\
    \geq \|\nabla r(x)\|^{-1}
    H_{r}\left(\xi^{T},\xi^{T}\right)(x) &+ K\left|\left<\nabla r_0(x),\xi\right>\right|^2,
\end{align}
if $x\in b\Omega$. Thus $r_1$ is a defining function for $\Omega$ which is convex on $b\Omega$.

\vskip.5cm
\begin{remark} If $\Omega\subset\subset\mathbb{R}^n$ is {\it strongly} convex, i.e., 
\begin{equation}\label{strict_convexdomain_def}
\sum_{j,k=1}^n\frac{\partial^2 r}{\partial x_j\partial x_k}(p)\xi_j\xi_k > 0\quad\sjump\forall\sjump p\in b\Omega,\quad\forall\sjump\xi\in T_{p}\left(b\Omega\right),
\end{equation}
for some defining function $r$,  then constructing a (strongly) convex defining function for $\Omega$ is simpler, essentially requiring only the second step above. Indeed, homogeneity, \eqref{strict_convexdomain_def}  and compactness of $b\Omega$ give a constant $c>0$ such that
$$ \sum_{j,k=1}^n\frac{\partial^2 r}{\partial x_j\partial x_k}(p)\xi_j\xi_k > c\left\|\xi\right\|^2 \quad\sjump\forall\sjump p\in b\Omega,\quad\forall\sjump\xi\in T_{p}\left(b\Omega\right).$$
Applying the Cauchy--Schwarz inequality to the mixed term in \eqref{bilinear}, then using the (sc)-(lc) inequality yields
$$ H_r\left(\xi,\xi\right) > \frac{c}{2}\left\|\xi^T\right\|^2 +\mathcal{O}\left(\left\|\xi^N\right\|^2\right).$$
Modifying $r$ as in subsection \ref{SS:controllingN} now allows the big-$\mathcal{O}$ term to be
positively absorbed. 
\end{remark}
\vskip.5cm

\subsection{Convexity in a neighborhood of $b\Omega$} Modifying $r_1$, to obtain a convex function off $b\Omega$, will also occur in two steps. The first step involves showing that, in tangential directions, $H_{r_1}$ has a lower bound near $b\Omega$ that is quadratic in the distance-to-the-boundary, but which may be {\it negative}. The following example illustrates that this negativity can occur.

\begin{example}
 Consider the function $s(x,y)=y+yx^2+x^4$ and the corresponding domain 
$D=\{(x,y)\in\mathbb{R}^{2}:s(x,y)<0\}$ near the origin.

Let $(a,b)\in bD$ and note that this implies $b=-\frac{a^{4}}{1+a^{2}}$. 
It follows from
$s_{x}(x,y)=4x^{3}+2yx$ and $s_{y}(x,y)=1+x^{2}$ that
\begin{align*}
  T_{(a,b)}(bD)
  =\{\tau\in\mathbb{R}^{2}: \frac{4a^{3}+2a^{5}}{1+a^{2}}\tau_{1}+(1+a^{2})\tau_{2}=0\}.
\end{align*}
Furthermore, a simple calculation yields
\begin{align*}
  s_{xx}(x,y)=12x^{2}+2y,\;\;s_{xy}(x,y)=2xy \;\text{ and }\; s_{yy}(x,y)=0.
\end{align*}
Hence, for $\tau\in T_{(a,b)}(bD)$ we obtain
\begin{align*}
  H_{s}(\tau,\tau)(a,b)&=(12a^{2}+2b)\tau_{1}^{2}+4ab\tau_{1}\tau_{2}\\
  &=\left(12a^{2}-\frac{2a^{4}}{1+a^{2}}+\frac{4a^{5}(4a^{3}+2a^{5})}{(1+a^{2})^{3}}\right)\tau_{1}^{2}\\
  &\geq 2a^{2}\left(6-\frac{a^{2}}{1+a^{2}}\right)\tau_{1}^{2}\geq 0.
\end{align*}
Thus $D$ is a convex domain. However, $s$ is not convex on $bD$ except at the origin. To see this, let
$\xi\in\mathbb{R}^{2}\setminus\{0\}$ and $(a,b)\in bD\setminus\{0\}$ and compute:
\begin{align*}
  H_{s}(\xi,\xi)(a,b)&=\left(12a^{2}-\frac{2a^{4}}{1+a^{2}}\right)\xi_{1}^{2}-\frac{4a^{5}}{1+a^{2}}\xi_{1}\xi_{2}\\
  &=\frac{2a^{2}}{1+a^{2}}\left( (6+5a^{2})\xi_{1}^{2}-2a^{3}\xi_{1}\xi_{2}\right),
\end{align*}
which is negative if, e.g., $\xi_{2}=\xi_{1}\cdot\frac{6+5a^{2}}{a^{3}}$.

\vskip0.5cm

\noindent From Section \ref{SS:convexonboundary} it follows 
that 
$$s_{1}=s_{0}+Ks_{0}^{2}=\frac{s}{\|\nabla s\|}+K\left(\frac{s}{\|\nabla s\|}\right)^{2}$$ 
is convex on $bD$ near the origin if the constant $K>0$ is chosen sufficiently large. To see that $s_{1}$ is not convex in any neighborhood of $0$, let us compute $H_{s_{1}}(\tau,\tau)(q)$ for $q=(0,\epsilon)$ and $\tau=(1,0)$, i.e.,
$\frac{\partial^{2} s_{1}}{\partial x^{2}}(q)$. First note that
\begin{align*}
  \frac{\partial^{2}}{\partial x^{2}}\bigl(s_{0}^{2}\bigr)=2s_{0}\frac{\partial^{2}s_{0}}{\partial x^{2}}
  +2\left(\frac{\partial s_{0}}{\partial x} \right)^{2}\;\;\text{and}\;\;\;
  \frac{\partial s_{0}}{\partial x}=\frac{s_{x}}{\|\nabla s\|}-\frac{s(s_{x}s_{xx}+s_{y}s_{yx})}{\|\nabla s\|^{3}}.
\end{align*}
Both of $s_{x}(q)$ and $s_{yx}(q)$ being $0$ causes $\frac{\partial s_{0}}{\partial x}(q)=0$. Since we also have $s(q)=\epsilon$, it follows that 
\begin{align*}
  \frac{\partial^{2}s_{1}}{\partial x^{2}}(q)=\left(1+2K\epsilon\right)\cdot\frac{\partial^{2}s_{0}}{\partial x^{2}}(q). 
 \end{align*}
Using again that $s_{x}(q)=0=s_{yx}(q)$, a straightforward computation yields
\begin{align*}
 \frac{\partial^{2}s_{0}}{\partial x^{2}}(q)&=
 \left(
 \frac{s_{xx}}{\|\nabla s\|}-\frac{s}{\|\nabla s\|^{3}}\left(s_{xx}^{2}+s_{y}s_{yxx}\right)
 \right)(q)\\
 &=2\epsilon-\epsilon\left((2\epsilon)^{2}+2 \right)=-4\epsilon^{3},
\end{align*}
because $s_{xx}(q)=2\epsilon$, $s_{y}(q)=1$, $\|\nabla s(q)\|=1$ and $s_{yxx}(q)=2$.
From $\delta(q)=\epsilon$ it follows that
$H_{s_{1}}(\tau,\tau)(q)=-4(1+2K\delta(q))\cdot\left(\delta(q)\right)^{3}$. Thus, moving from $s$ to $s_{1}$ does not achieve convexity at $q$ with $\delta(q)>0$. However, it improves the lower bound on the Hessian in the tangent direction $\tau$ as $H_{s}(\tau,\tau)(q)=s_{xx}(q)=2\delta(q)$.
\end{example}

\vskip.7cm

\subsubsection{Quadratic estimate in tangent directions}

The estimate we desire follows by combining the fact that $\|\nabla r_{1}\|=1$ on $b\Omega$ with Taylor's theorem.

\begin{proposition}\label{P:quadratic}
  Let $\Omega\subset\mathbb{R}^{n}$ be a smoothly bounded domain and $p\in b\Omega$. Suppose 
  $\sigma$ is a smooth defining function for $\Omega$ in a neighborhood $U\subset\mathbb{R}^{n}$ of 
  $p$ satisfying $\|\nabla\sigma\|=1$ on $b\Omega\cap U$.
  
  If $\Omega$ is convex on $U$, then there exist a neighborhood $V\subset U$ of $p$ and constants
  $C_{1},C_{2}>0$ such that
  \begin{align}\label{E:lowerboundsigma}
    H_{\sigma}(\xi,\xi)(x)\geq -C_{1}\left(\sigma^{2}(x)\right)\|\xi\|^{2}
    -C_{2}\frac{\left|\langle\nabla\sigma(x),\xi\rangle\right|^{2}}{\|\nabla\sigma(x)\|^{2}}
  \end{align}
  holds for all $x\in V$ and $\xi\in\mathbb{R}^{n}$.
\end{proposition}

The following lemma is used in the proof of Proposition \ref{P:quadratic}.

\begin{lemma}\label{L:sigmalemma}
  Suppose the hypotheses of Proposition \ref{P:quadratic} hold. Then
  \begin{align}\label{E:claim}
    \Bigl\langle\nabla H_{\sigma}(\tau,\tau),\nabla\sigma\Bigr\rangle(x)=
    \Bigl(
    H_{\sigma}(\tau,\tau)\cdot H_{\sigma}\bigl(\nabla\sigma,\nabla\sigma\bigr)
    -H_{\sigma}\bigl(\tau,\mathcal{H}_{x}^{\sigma}\tau\bigr)
    \Bigr)(x)
  \end{align}
  for all $x\in b\Omega\cap U$ and $\tau\in T_{x}(b\Omega)$.
\end{lemma}

\begin{proof}[Proof of Lemma \ref{L:sigmalemma}]
  Differentiating  $\|\nabla\sigma\|^{2}=1$ at $x\in b\Omega\cap U$ in  tangential directions yields,
  as in \ref{SS:controllingM},
   \begin{align}\label{E:mixedtermssigma}
     0=H_{\sigma}\left(\tau,\nabla\sigma\right)(x)
     \qquad\sjump\forall\sjump \tau\in T_{x}(b\Omega).
   \end{align}
   Whereas in the case of $\sigma=\delta$ the equation \eqref{E:mixedtermssigma} is true for \emph{all} 
   directions  (see \eqref{E:pairingnabladelta}), 
   \eqref{E:mixedtermssigma} holds only for tangential directions. As a result we use tangential 
   vector fields  to obtain  
   information on the desired third order derivatives of 
   $\sigma$.      
   
   Let $x_{0}\in b\Omega\cap U$ and $\tau\in T_{x_{0}}(b\Omega)$ be given.  Since $\nabla\sigma$ is 
   non-vanishing on $b\Omega\cap U$, it may be assumed that $\sigma_{x_{n}}$ is 
   non-zero in a neighborhood 
   $V\subset U$ of $x_0$. Let $e_1,\dots,e_n\in\mathbb{R}^n$ be the standard basis of 
   $\mathbb{R}^n$ and define
   \begin{align*}
     T^i(x)=\sigma_{x_n}(x)e_i-\sigma_{x_i}(x)e_n\qquad\forall\sjump i\in\{1,\dots,n-1\}.
   \end{align*}
   At each point $x\in b\Omega\cap V$ the vectors $T^1(x),\dots,T^{n-1}(x)$ form a 
   basis of $T_x(b\Omega)$.  Therefore, constants $a_{1},\dots,a_{n-1}\in\mathbb{R}$ can be chosen such that
   the tangential vector field $T(x):=\sum_{i=1}^{n-1}a_{i}T^{i}(x)$ equals $\tau$ when evaluated at $x_{0}$.
   Differentiating $\|\nabla\sigma\|_{|_{b\Omega\cap V}}^{2}=1$ in the direction of $T(x)$ gives us
   \begin{align*}
      H_\sigma\left(T,\nabla\sigma\right)(x)=0\qquad\forall\sjump x\in b\Omega\cap V,
   \end{align*}
   and differentiating this last equation again with respect  to $T$ at $x_{0}$ yields
   \begin{align*}
       \Bigl\langle\nabla\left(H_\sigma\bigl(T,\nabla\sigma\bigr)\right),T
       \Bigr\rangle  (x_0)
       =\Bigl\langle\nabla\left(H_{\sigma}\bigl(T,\nabla\sigma\bigr)\right),\tau
       \Bigr\rangle(x_{0})=0.
  \end{align*}
  It follows from a straightforward computation that
  \begin{align}\label{E:thirdderivativessigmaid}
    0=&\left\langle \nabla H_\sigma\left(\tau,\tau\right),\nabla\sigma\right\rangle(x_0)
    +H_{\sigma}\bigl(X,\nabla\sigma\bigr)(x_{0})
    +H_{\sigma}\bigl(\tau,\mathcal{H}_{x_{0}}^{\sigma}\tau \bigr)(x_{0})\\
       &\hspace{1cm}\text{for }\;\;X:=\sum_{i=1}^{n-1}a_{i}\Bigl(\langle\nabla\sigma_{x_{n}},
       \tau\rangle(x_{0})e_{i}
       -\langle\nabla\sigma_{x_{i}},\tau\rangle(x_{0})e_{n}\Bigr).\notag
  \end{align}
  For $H_{\sigma}(X,\nabla\sigma)(x_{0})$, note that
  \eqref{E:mixedtermssigma}  implies that
  \begin{align*}
    H_{\sigma}\left(X,\nabla\sigma\right)(x_{0})=H_{\sigma}\left(X^{N},\nabla\sigma\right)(x_{0})\;\;\text{for}\;\;
    X^{N}=\frac{\langle\nabla\sigma(x_{0}),X\rangle}{\|\nabla\sigma(x_{0})\|^{2}}\cdot\nabla\sigma(x_{0}).
  \end{align*}
  In the special case of $\sigma=\delta$ the term $H_{\delta}(X,\nabla\delta)(x_{0})$ vanishes since $\delta$ satisfies  
  \eqref{E:pairingnabladelta}. For general $\sigma$ compute
   \begin{align*}
    X=\sum_{i=1}^{n-1}a_{i}\Bigl(
    \left\langle\mathcal{H}_{x_{0}}^{\sigma}\tau,e_{n}\right\rangle\cdot e_{i}
    -\left\langle\mathcal{H}_{x_{0}}^{\sigma}\tau,e_{i}\right\rangle\cdot e_{n}
    \Bigr),\notag
  \end{align*}
  to obtain
  \begin{align*}
    X^{N}&=
    \sum_{i=1}^{n-1}\Bigl(\left\langle\mathcal{H}_{x_{0}}^{\sigma}\tau,a_{i}\sigma_{x_{i}}(x_{0})\cdot e_{n}
    \right\rangle
    -\left\langle\mathcal{H}_{x_{0}}^{\sigma}\tau,a_{i}\sigma_{x_{n}}(x_{0})\cdot e_{i}
    \right\rangle\Bigr)\cdot\nabla\sigma(x_{0})\\
    &=-\left\langle\mathcal{H}_{x_{0}}^{\sigma}\tau,\tau\right\rangle\cdot\nabla\sigma(x_{0})
    =-H_{\sigma}\left(\tau,\tau\right)(x_{0})\cdot\nabla\sigma(x_{0}).
  \end{align*}
  This implies that
  \begin{align*}
   H_{\sigma}\left(X,\nabla\sigma)(x_{0}\right)
   =-\bigl(H_{\sigma}\left(\tau,\tau\right)\cdot H_{\sigma}\left(\nabla\sigma,\nabla\sigma\right)\bigr)(x_{0}).
  \end{align*}
  The last equation combined with \eqref{E:thirdderivativessigmaid} proves the identity \eqref{E:claim}. 
\end{proof}

\begin{proof}[Proof of Proposition \ref{P:quadratic}] 
  Shrink the neighborhood $U\subset\mathbb{R}^{n}$ of the given boundary point $p$ such that $b_{b\Omega}=b$ is well defined on $U$, 
  cf. Lemma \ref{L:facts} (a).
  For $f\in C^{\infty}(U,\mathbb{R})$ and $x\in U$, it then follows from Taylor's theorem  that
  \begin{align*}
    f(x)=f\bigl(b(x)\bigr)+\Bigl\langle\left(\nabla f\right)\bigl(b(x)\bigr),x-b(x)\Bigr\rangle
    +\frac{1}{2}H_{f}(x-b(x),x-b(x))(\zeta)
  \end{align*}
  holds for some point $\zeta$ on the line segment connecting $x$ and $b(x)$.
  By \eqref{E:bidentity} and \eqref{E:deltanu} it follows that
   $x-b(x)=\delta(x)\nabla\delta\bigl(b(x)\bigr)$ and $\nabla\delta\left(b(x)\right)=\nabla\delta(\zeta)$, respectively. Therefore
   \begin{align}\label{E:Taylorf}
    f(x)&=
    f\bigl(b(x)\bigr)+\delta(x)\left\langle\nabla f, \nabla\delta\right\rangle\bigl(b(x)\bigr)
    +\frac{1}{2}\left(\delta(x)\right)^{2}H_{f}\left(\nabla\delta,\nabla\delta\right)(\zeta)\notag\\
    &\geq f\bigl(b(x)\bigr)+\delta(x)\left\langle\nabla f,\nabla\delta\right\rangle\bigl(b(x)\bigr)
    - c\left(\delta(x)\right)^{2}
   \end{align} 
   for some constant $c>0$.
   
   Let $x_{0}\in U$, with $y_{0}:=b(x_{0})\in b\Omega\cap U$, and $\xi\in\mathbb{R}^{n}$ be given. 
   Since $\|\nabla\sigma(y_{0})\|=1$, 
   it follows that $\nabla\sigma(y_{0})=\nabla\delta(y_{0})$. Thus, applying \eqref{E:Taylorf} to the function
   $H_{\sigma}(\xi,\xi)(x)$ at $x_{0}$ results in
   \begin{align*}
    H_{\sigma}(\xi,\xi)(x_0)\geq H_\sigma(\xi,\xi)(y_0)+\delta(x_0)
    \left\langle\nabla H_{\sigma}(\xi,\xi),\nabla\sigma\right\rangle(y_0)
   -c_{1}(\delta(x_0))^2\|\xi\|^2
  \end{align*}
  for some constant $c_{1}>0$, independent of $x_{0}$ and $\xi$.
  For $\xi=\xi_{y_{0}}^{T}+\xi_{y_{0}}^{N}$ it follows from \eqref{E:mixedtermssigma} that
  \begin{align*}
    H_\sigma(\xi,\xi)(y_0)=H_\sigma(\xi^{T},\xi^{T})(y_0)+H_\sigma(\xi^{N},\xi^{N})(y_0).
  \end{align*}
  Also,
  \begin{align*}
    \left\langle\nabla H_\sigma(\xi,\xi),\nabla\sigma\right\rangle(y_0)=
    \left\langle\nabla H_\sigma\left(\xi^T,\xi^T\right),\nabla\sigma\right\rangle(y_0)
    +\mathcal{O}\left(\left\|\xi^T\right\|\cdot \left\|\xi^N\right\|,\left\|\xi^N\right\|^2\right).
  \end{align*}
  Using the (sc)-(lc) inequality then implies the existence of  constants $c_{2}, c_{3}>0$ such that
  \begin{align}\label{E:Taylorsigmaimproved}
    H_\sigma(\xi,\xi)(x_0)\geq
    H_\sigma(\xi^{T},\xi^{T})(y_0)+&\delta(x_0)\left\langle\nabla H_\sigma\left(\xi^T,\xi^T\right),\nabla\sigma 
    \right\rangle(y_0)\\
    &\hspace{2cm}-c_2\left(\delta(x_0)\right)^2\left\|\xi^T\right\|^2-c_3\left\|\xi^N\right\|^2,\notag
  \end{align}
  where $\|\xi\|^{2}=\|\xi^{T}\|^{2}+\|\xi^{N}\|^{2}$ is used. It now follows from Lemma \ref{L:sigmalemma} that
  \begin{align*}
    H_\sigma(\xi,\xi)(x_0)\geq&
    H_\sigma(\xi^{T},\xi^{T})(y_0)\cdot
    \Bigl(1+\delta(x_{0})H_{\sigma}(\nabla\sigma,\nabla\sigma)(y_{0})\Bigr)\\
    &-\delta(x_{0})H_{\sigma}(\xi^{T},\mathcal{H}_{y_{0}}^{\sigma}\xi^{T})(y_{0})
    -c_2\left(\delta(x_0)\right)^2\left\|\xi^T\right\|^2-c_3\left\|\xi^N\right\|^2.
  \end{align*}
  Since $\xi^{T}\in T_{y_{0}}(b\Omega)$, \eqref{E:mixedtermssigma} implies that 
  $\mathcal{H}_{y_{0}}^{\sigma}\xi^{T}\in T_{y_{0}}(b\Omega)$. Thus,  as 
  $\Omega$ is convex near $y_{0}$, the Cauchy--Schwarz inequality can be applied
   \begin{align*}
     \left|H_{\sigma}\left(\xi^{T},\mathcal{H}_{y_{0}}^{\sigma}\xi^{T}\right)(y_{0})\right|
     &\leq
     \left(H_{\sigma}\left(\xi^{T},\xi^{T}\right)(y_{0})\right)^{\frac{1}{2}}\cdot
     \Bigl(H_{\sigma}\left(\mathcal{H}_{y_{0}}^{\sigma}\xi^{T},
     \mathcal{H}_{y_{0}}^{\sigma}\xi^{T}\right)(y_{0})\Bigr)^{\frac{1}{2}}\\
     &\leq
     \sqrt{c_{4}}\cdot\left(H_{\sigma}\left(\xi^{T},\xi^{T}\right)(y_{0})\right)^{\frac{1}{2}}\cdot
     \left\|\mathcal{H}_{y_{0}}^{\sigma}\xi\right\|
     \end{align*}
     for  $c_{4}\cdot\max\{|\sigma_{x_{i}x_{j}}(x)|:1\leq i,j\leq n, \;x\in b\Omega\cap U\}.$ Note that since
      \begin{align*}
        H_{\sigma}\left(\xi^{T},\mathcal{H}_{y_{0}}^{\sigma}\xi^{T}\right)(y_{0})
        =\|\mathcal{H}_{y_{0}}^{\sigma}\xi\|^{2}
      \end{align*} 
      it follows from the (sc)-(lc) inequality that
       \begin{align*}
     0\leq H_{\sigma}\left(\xi^{T},\mathcal{H}_{y_{0}}^{\sigma}\xi^{T}\right)(y_{0})
     \leq c_{4}\cdot H_{\sigma}\left(\xi^{T},\xi^{T}\right)(y_{0}).
   \end{align*}
   Therefore
   \begin{align*}
     H_{\sigma}(\xi,\xi)(x_{0})\geq
     \left(1-2c_{4}\left|\delta(x_{0})\right|\right)\cdot H_{\sigma}\left(\xi^{T},\xi^{T}\right)(y_{0})
     -c_{2}(\delta^{2}(x_{0})\left\|\xi^{T}\right\|^{2}
     -c_{3}\left\|\xi^{N}\right\|^{2}.
   \end{align*}
   Shrink $U$ so that
   $2c_{4}|\delta(x)|\leq 1$ for all $x\in U$. Furthermore, note that there exists a constant $c_{5}>0$ such that
   $\delta^{2}(x)\leq c_{5}\sigma^{2}(x)$ for all $x\in U$. It then follows that
   \begin{align*}
     H_{\sigma}(\xi,\xi)(x_{0})
     \geq
     -c_{2}c_{5}\sigma^{2}(x_{0})\left\|\xi\right\|^{2}
     -c_{3}\left\|\xi^{N}\right\|^{2}.
   \end{align*}
   Since $\xi^{N}=\frac{\langle\nabla\sigma(y_{0}),\xi\rangle}{\|\nabla\sigma(y_{0})\|^{2}}\nabla\sigma(y_{0})$, it 
   follows from Taylor's Theorem that there exists a constant $c_{6}>0$ such that
   \begin{align*}
     \left\|\xi^{N}\right\|^{2}\leq 2\frac{|\langle\nabla\sigma(x_{0}),\xi\rangle|^{2}}{\|\nabla\sigma(x_{0})\|^{2}}   
     +c_{6}\sigma^{2}(x_{0})\|\xi\|^{2},
   \end{align*}
   which completes the proof.
  \end{proof}

\subsubsection{Quadratic negativity implies ``free'' positivity} It remains to show how to pass from the estimate \eqref{E:lowerboundsigma} 
to convexity of a related defining function.

\begin{proposition} Let $\sigma$ satisfy the conditions in Proposition \ref{P:quadratic}. For any $\alpha,\beta>0$, define the function $B(x)=B_{\alpha,\beta}(x)=\alpha +\beta\|x\|^2$ for $x\in\mathbb{R}^n$.

Then there exist a neighborhood $W\subset\mathbb{R}^{n}$ of $p$ and constants $\alpha, \beta$ such that the defining function
$$\tilde\sigma(x)=\tilde\sigma_{\alpha,\beta}(x)=\sigma(x)+ B(x)\cdot\sigma^2(x)$$
satisfies
$$H_{\tilde\sigma}\left(\xi,\xi\right)(x)\geq 0\quad\sjump\forall\sjump x\in W,\quad\forall\sjump\xi\in \mathbb{R}^n.$$
\end{proposition}

\begin{proof} 
Let $\alpha,\beta>0$ be fixed constants to be chosen later. Let $V\subset\mathbb{R}^{n}$ be a neighborhood of $p$ such that \eqref{E:lowerboundsigma} holds for $\sigma$ on $V$.
For $x\in V$ and $\xi\in\mathbb{R}^{n}$ a simple calculation yields
\begin{align*}
  H_{\tilde{\sigma}}(\xi,\xi)(x)=H_{\sigma}(\xi,\xi)(x)+H_{B\cdot\sigma^{2}}(\xi,\xi)(x),
\end{align*}
where
\begin{align*}
  H_{B\cdot\sigma^{2}}(\xi,\xi)(x)
  =
  &B(x)\cdot\Bigl(2\sigma(x)\cdot H_{\sigma}(\xi,\xi)(x)
  +2\left|\left\langle\nabla\sigma(x),\xi\right\rangle\right|^{2}\Bigr)(x)\\
  &+8\beta\sigma(x)\left\langle x,\xi\right\rangle\langle\nabla\sigma(x),\xi\rangle
  +2\beta\left(\sigma(x)\right)^{2}\|\xi\|^{2}.
\end{align*}
By the  Cauchy--Schwarz inequality and the (sc)-(lc) inequality it follows that
\begin{align*}
  \left| 8\beta\sigma(x)\left\langle x,\xi\right\rangle\left\langle\nabla\sigma(x),\xi\right\rangle\right|
  \leq
  16\beta\|x\|^{2}\left|\left\langle\nabla\sigma(x),\xi\right\rangle\right|^{2}
  +\beta\left(\sigma(x)\right)^{2}\|\xi\|^{2},
\end{align*}
which implies that
\begin{align*}
  H_{\tilde{\sigma}}(\xi,\xi)(x)
  \geq
  \bigl(1+2B(x)\sigma(x)\bigr)\cdot H_{\sigma}(\xi,\xi)(x)
  &+2\left(\alpha-7\beta\|x\|^{2}\right)\left|\left\langle\nabla\sigma(x),\xi\right\rangle\right|^{2}\\
  &+\beta\left(\sigma(x)\right)^{2}\|\xi\|^{2}.
\end{align*}
Choose a neighborhood $W=W(\alpha,\beta)\subset V$ of $p$ such that $0\leq 1+2B(x)\sigma(x)\leq 2$ 
for all $x\in W$. Then it follows from \eqref{E:lowerboundsigma} that there exist constants $C_{1},C_{2}>0$ such that
\begin{align*}
  H_{\tilde{\sigma}}(\xi,\xi)(x)
  \geq& -\bigl(1+2B(x)\sigma(x)\bigr)\left(C_{1}\sigma^{2}(x)\|\xi\|^{2}
  +C_{2}\frac{|\langle\nabla\sigma(x),\xi\rangle|^{2}}{\|\nabla\sigma(x)\|}\right)\\
  &+2\left(\alpha-7\beta\|x\|^{2}\right)\left|\left\langle\nabla\sigma(x),\xi\right\rangle\right|^{2}
  +\beta\left(\sigma(x)\right)^{2}\|\xi\|^{2}
 \end{align*}
 for all $x\in W$ and $\xi\in\mathbb{R}^{n}$. Since $1+B\cdot\sigma\leq 2$ on $W$ it follows that
 \begin{align*} 
  H_{\tilde{\sigma}}(\xi,\xi)(x)
   \geq &\bigl(-2C_{1}+\beta\bigr)\cdot\sigma^{2}(x)\|\xi\|^{2}\\
  &+2\Bigl(-\frac{C_{2}}{\|\nabla\sigma(x)\|^{2}}+\alpha-7\beta\|x\|^{2}\Bigr)\cdot
  |\langle\nabla\sigma(x),\xi\rangle|^{2}
\end{align*}
for all $x\in W$ and $\xi\in\mathbb{R}^{n}$. Set $m_{1}=\max\{\|x\|^{2} : x\in V\}$. Furthermore, let $m_{2}>0$ be such that 
$m_{2}\leq \|\nabla\sigma(x)\|^{2}$ for $x\in V$.
By choosing $\beta\geq 2C_{1}$ and then $\alpha\geq 7\beta\cdot m_{1}+\frac{C_{2}}{m_{2}}$ it follows that 
$\tilde{\sigma}_{\alpha,\beta}$ is convex on $W$.
\end{proof}

\section{ Logarithmic convexity}\label{S:logconvex}

Let $\delta=\delta_{b\Omega}$ denote the signed distance-to-the-boundary function associated to a smoothly bounded domain $\Omega\subset\mathbb{R}^{n}$, as in Section \ref{S:delta}. 
The fact that $\left\|\nabla\delta\right\|\equiv 1$, in a neighborhood of $b\Omega$, implies the following curious result.

\medskip
\begin{proposition}\label{logconvexdelta}
Let $\Omega\subset\mathbb{R}^{n}$ be a smoothly bounded domain. Let $U$ denote a sufficiently small neighborhood of $b\Omega$ such that $\delta$ is smooth on $U$. 
Then $$-\log(-\delta) \text{ is convex on }\Omega\cap U\Longleftrightarrow \delta\text{ is convex on }\Omega\cap U.$$
\end{proposition}
\medskip

The left to right implication is surprising: for a general smooth, negative function $f$, the convexity of $-\log\left(-f\right)$ does {\it not} imply that $f$ is convex; cf. \eqref{hessian_chiOr} and note that $\psi(x)=-e^{-x}$ on $(-\infty, 0)$ is not convex. This implication can be rewritten in terms of the Hessian of $\delta$ alone: if the negative lower bound
\begin{equation}\label{loghessian_rewrite}
H_{\delta}(\xi,\xi)(x)\geq -\frac{\left|\left<\nabla\delta(x),\xi_x^N\right>\right|^2}{-\delta(x)},\qquad x\in U\cap\Omega, \xi\in\mathbb{R}^n
\end{equation} holds, then necessarily
\begin{equation*}
H_{\delta}(\xi,\xi)(x)\geq 0,\qquad x\in U\cap\Omega, \xi\in\mathbb{R}^n.
\end{equation*} Thus, as soon as a ``threshold of negativity'' (the right hand side of \eqref{loghessian_rewrite}) is exceeded, the Hessian of $\delta$ is actually non-negative. 

The elementary nature of Proposition \ref{logconvexdelta} suggests this result may be known. But the authors were unable to find a statement or proof of this fact in the literature, so offer one here.

\begin{proof} Since $\left\|\nabla\delta\right\|\equiv 1$ in $U$, for any $\xi\in\mathbb{R}^n$

\begin{equation}\label{hessian_nomixed}
H_{\delta}(\xi,\nabla\delta)(x)=0\;\;\text{for all}\;\;x\in U;
\end{equation}
see \eqref{E:pairingnabladelta}. The expansion \eqref{bilinear} then yields

\begin{equation}\label{hessian_tang_delta}
H_{\delta}(\xi,\xi)(x)= H_{\delta}\left(\xi_x^{T},\xi_x^{T}\right)(x),\qquad\text{for }x\in U,
\end{equation} 
where $\xi=\xi_x^T+\xi_x^N$ is the splitting introduced at the beginning of Section \ref{S:other}. It also follows from \eqref{hessian_chiOr} and \eqref{hessian_nomixed} that

\begin{equation}\label{hessian_tang_logdelta}
H_{-\log(-\delta)}(\xi,\xi)(x)= \frac 1{-\delta(x)}H_{\delta}\left(\xi_x^{T},\xi_x^{T}\right)(x) +\frac 1{\delta^2(x)}\left|\left<\nabla\delta(x),\xi_x^N\right>\right|^2.
\end{equation}
\medskip

$\left(\Longleftarrow\right)$ Assume that $\delta$ is convex, i.e., $H_{\delta}(\xi,\xi)(x)\geq 0$ for $x\in U$, all $\xi\in\mathbb{R}^n$.
Then \eqref{hessian_tang_delta} and \eqref{hessian_tang_logdelta} immediately imply that 
$H_{-\log(-\delta)}(\xi,\xi)(x)\geq 0$.
\medskip

$\left(\Longrightarrow\right)$ Assume that$-\log(-\delta)$ is convex in $U$. Then in particular \newline $H_{-\log(-\delta)}(\tau,\tau)(x)\geq 0$ for $x\in U$ and any direction $\tau\in T_x\left(\left\{\delta =\delta(x)\right\}\right)$. For such $\tau$, the second term on the right-hand side of \eqref{hessian_tang_logdelta} vanishes, and we obtain
$$H_{-\log(-\delta)}(\tau,\tau)(x)= \frac 1{-\delta(x)}H_{\delta}\left(\tau,\tau\right)(x).$$
Thus, the fact that $H_{-\log(-\delta)}(\tau,\tau)(x)\geq 0$ forces $H_{\delta}\left(\tau,\tau\right)(x)\geq 0$ for such directions $\tau$. Since \eqref{hessian_tang_delta} holds, the convexity of $\delta$ is demonstrated.
    
\end{proof}

The fact that $\nabla\delta$ has constant length in a {\it neighborhood} of $b\Omega$ was essential for the proof of Proposition  \ref{logconvexdelta}. However, defining functions whose gradients have constant length only {\it on} $b\Omega$ also force similar statements relating lower bounds on their Hessians to lower bounds on the Hessian of their logarithms. 

\begin{proposition}
  Let $\sigma$ be a defining function for a smoothly bounded domain $\Omega\subset\mathbb{R}^{n}$ with 
  the property $\|\nabla\sigma\|=1$ on $b\Omega$. Then
  \begin{align}\label{E:Hessianlogsigma}
    H_{-\log(-\sigma)}(\xi,\xi)(x)\geq -C|\sigma(x)|\|\xi\|^{2}\qquad\forall x\in \Omega\cap U,
    \sjump\xi\in\mathbb{R}^{n}
  \end{align}
  if and only if  
  \begin{align}\label{E:Hessiansigma}  
    H_{\sigma}(\xi,\xi)(x)\geq -\tilde{C}\left(\sigma^{2}(x)\|\xi\|^{2}
    +\frac{|\langle\nabla\sigma(x),\xi\rangle|^{2}}{\|\nabla\sigma(x)\|^{2}}\right)\qquad\forall x\in \Omega\cap 
    \tilde{U},
    \sjump\xi\in\mathbb{R}^{n},
  \end{align}
  where $U$ and $\tilde{U}$ are neighborhoods of $b\Omega$ and $C,\tilde{C}>0$.
\end{proposition}

\begin{proof}
  $\left(\Longleftarrow\right)$ Assume that \eqref{E:Hessiansigma} holds for some neighborhood $\tilde{U}$ of $b\Omega$ and some constant $\tilde{C}>0$. Applying \eqref{E:Hessiansigma}, followed by trivial estimation, yields
  \begin{align*}
    H_{-\log(-\sigma)}(\xi,\xi)(x)&=\frac{1}{-\sigma(x)}H_{\sigma}(\xi,\xi)(x)
    +\frac{1}{\sigma^{2}(x)}|\nabla\sigma(x),\xi\rangle|^{2}\\
    &\geq -\tilde C\left(
    |\sigma(x)|\|\xi\|^{2}+\frac{|\langle\nabla\sigma(x),\xi\rangle|^{2}|}{|\sigma(x)|\cdot\|\nabla\sigma(x)\|^{2}}
    \right)+\frac{|\langle\nabla\sigma(x),\xi\rangle|^{2}}{\sigma^{2}(x)}\\
    &\geq
    -\tilde{C}|\sigma(x)|\|\xi\|^{2}
  \end{align*}
  for all $x\in\Omega\cap U$, where $U\subset\tilde{U}$ is a sufficiently small neighborhood of $b\Omega$.
  \medskip
  
  $\left(\Longrightarrow\right)$ Now suppose that \eqref{E:Hessianlogsigma} 
  holds for some neighborhood $U$ 
  of $b\Omega$ and some constant $C>0$. For given $x\in\Omega\cap U$ and $\xi\in\mathbb{R}^{n}$, write
  $\xi=\xi_{x}^{T}+\xi_{x}^{N}$. The fact that $\|\nabla\sigma\|=1$ on $b\Omega$ and Taylor's theorem give 
  the estimate
  \begin{align*}
    \left|H_{\sigma}\left(\xi_{x}^{T},\xi_{x}^{N} \right)(x) \right|=\mathcal{O}(|\sigma(x)|)\left\|\xi_{x}^{T}\right\| 
    \cdot
    \left\|\xi_{x}^{N}\right\|
  \end{align*}
  while trivial estimates show
  \begin{align*}
    \left|H_{\sigma}\left(\xi_{x}^{N},\xi_{x}^{N}
    \right)
    \right|
    =\mathcal{O}(1)\left\|\xi_{x}^{N}\right\|^{2}=
    \mathcal{O}(1)\frac{|\langle\nabla\sigma(x),\xi\rangle|^{2}}{\|\nabla\sigma(x)\|^{2}}.
  \end{align*}
  Using first \eqref{E:Hessianlogsigma} for $\xi_{x}^{T}$, followed by these two estimates, gives
  \begin{align*}
    H_{\sigma}(\xi,\xi)(x)&=H_{\sigma}\left(\xi_{x}^{T},\xi_{x}^{T}\right)(x)
    +2H_{\sigma}\left(\xi_{x}^{T},\xi_{x}^{N}\right)(x)+H_{\sigma}\left(\xi_{x}^{N},\xi_{x}^{N}\right)(x)\\
    &\geq -C\sigma^{2}(x)\left\|\xi_{x}^{T}\right\|^{2}
    +2H_{\sigma}\left(\xi_{x}^{T},\xi_{x}^{N}\right)(x)+H_{\sigma}\left(\xi_{x}^{N},\xi_{x}^{N}\right)(x)\\
    &\geq -C\sigma^{2}(x)\left\|\xi_{x}^{T}\right\|^{2} -K_{1}|\sigma(x)|\left\|\xi_{x}^{T}\right\|
    \left\|\xi_{x}^{N}\right\|-K_{2}\frac{|\langle\nabla\sigma(x),\xi\rangle|^{2}}{\|\nabla\sigma(x)\|^{2}}
  \end{align*}
  for some constants $K_{1},K_{2}>0$. It then follows from the (sc)-(lc) inequality and 
  $\|\xi\|^{2}=\|\xi_{x}^{T}\|^{2}+\|\xi_{x}^{N}\|^{2}$ that
  \begin{align*}
    H_{\sigma}(\xi,\xi)(x)\geq
    -\tilde{C}\left(\sigma^{2}(x)\|\xi\|^{2}
    +\frac{|\langle\nabla\sigma(x),\xi\rangle|^{2}}{\|\nabla\sigma(x)\|^{2}}\right)
  \end{align*}
  for some constant $\tilde{C}>0$.
\end{proof}

\bibliographystyle{plain}
\bibliography{HerMcN09}  

\begin{thebibliography}{1}

\bibitem{behrens84}
Mechthild Behrens.
\newblock Plurisubharmonic defining functions of weakly pseudoconvex domains in
  {${\bf C}\sp 2$}.
\newblock {\em Math. Ann.}, 270(2):285--296, 1985.

\bibitem{diederichfornaess77}
Klas Diederich and John~Erik Forn{\ae}ss.
\newblock Pseudoconvex domains: an example with nontrivial {N}ebenh\"ulle.
\newblock {\em Math. Ann.}, 225(3):275--292, 1977.

\bibitem{federer59}
Herbert Federer.
\newblock Curvature measures.
\newblock {\em Trans. Amer. Math. Soc.}, 93:418--491, 1959.

\bibitem{fornaess79}
John~Erik Forn{\ae}ss.
\newblock Plurisubharmonic defining functions.
\newblock {\em Pacific J. Math.}, 80(2):381--388, 1979.

\bibitem{GilbargTrudinger}
David Gilbarg and Neil Trudinger.
\newblock {\em Elliptic partial differential equations of second order, 2nd
  Ed.}, volume 224 of {\em Grundlehren der Mathematischen Wissenschaften}.
\newblock Springer-Verlag, New York, 1983.

\bibitem{hormander_convex_book}
Lars H\"ormander.
\newblock {\em Notions of Convexity}, volume 127 of {\em Progress in
  Mathematics}.
\newblock Birkh\"auser, 1994.

\bibitem{krantz_scv_book}
Steven~G. Krantz.
\newblock {\em Function Theory of Several Complex Variables, 2nd Ed.}
\newblock Wadsworth \& Brooks/Cole Mathematics Series. Wadsworth \&
  Brooks/Cole, 1992.

\bibitem{Weinstock75}
Barnet~M. Weinstock.
\newblock Some conditions for uniform {$H$}-convexity.
\newblock {\em Illinois J. Math.}, 19:400--404, 1975.

\end{thebibliography}
\end{document}